\documentclass[reqno]{amsart}
\usepackage[utf8]{inputenc}
\usepackage[T1]{fontenc}
\usepackage[english]{babel}
\usepackage{mathtools}
\usepackage{amssymb}
\usepackage{amsthm}
\usepackage[arrow, matrix, curve]{xy}
\usepackage{graphicx}
\usepackage{tikz}
\usepackage{indentfirst}
\usepackage{MnSymbol}
\usepackage{bbm}
\usepackage{tikz-cd}
\usepackage{amssymb,fge}
\usepackage{enumitem}

\mathtoolsset{showonlyrefs,showmanualtags}

\newtheorem{Satz}{Satz}[section]
\newtheorem{Lem}[Satz]{Lemma}
\newtheorem{Ex}[Satz]{Example}

\newtheorem{Cor}[Satz]{Corollary}
\newtheorem{Pro}[Satz]{Proposition}
\newtheorem{Thm}[Satz]{Theorem}

\newcommand{\W}{d_W}
\newcommand{\N}{\mathbb{N}}

\newcommand{\R}{\mathbb{R}}

\newcommand{\inv}{^{-1}}
\newcommand{\A}{\mathcal{A}}

\newcommand{\Leb}{\mathcal{L}}
\newcommand{\Ha}{\frac{1}{2}}
\newcommand{\Hs}{H^2}

\newcommand{\Ds}{d_{S^2}}
\newcommand{\Dk}{d_{S^1}}

\newcommand{\K}{{S^1}}

\newcommand{\D}{D^2}
\newcommand{\Pih}{\frac{\pi}{2}}

\newcommand{\Ja}{\tn{\tb{J}}}

\newcommand{\tn}[1]{\textnormal{#1}}
\newcommand{\tb}[1]{\textit{#1}}
\newcommand{\tr}{\textrm{d}}
\makeatletter

\newcommand{\Rom}[1]{\expandafter\@slowromancap\romannumeral #1@}
\makeatother

\title{Majorization by Hemispheres \& Quadratic Isoperimetric Constants}
\author{Paul Creutz}
\address{Paul Creutz, Mathematisches Institut der Universit\"at zu K\"oln, Weyertal 86-90, 50931 K\"oln, Germany}
\email{pcreutz@math.uni-koeln.de}
\thanks{The author was partially supported by the DFG grant SPP 2026.}
\begin{document}
\begin{abstract}
Let $X$ be a Banach space or more generally a complete metric space admitting a conical geodesic bicombing. We prove that every closed $L$-Lipschitz curve $\gamma:\K\rightarrow X$ may be extended to an L-Lipschitz map defined on the hemisphere $f:H^2\rightarrow X$. This implies that $X$ satisfies a quadratic isoperimetric inequality (for curves) with constant $\frac{1}{2\pi}$. We discuss how this fact controls the regularity of minimal discs in Finsler manifolds when applied to the work of Alexander Lytchak and Stefan Wenger.
\end{abstract}
\maketitle
\section{Introduction}
\subsection{A Lipschitz Extension Theorem}
The famous majorization theorem of Reshetnyak states that for every rectifiable closed curve $\eta$ in a Hadamard space~$X$ there exists a convex domain $C\subseteq \R^2$ and a $1$-Lipschitz map $f:C \rightarrow X$ such that $f$ restricts to a length preserving parametrization of $\eta$ on $\partial C$, see \cite{Res68}. The main result of this paper is the following spherical analog of Reshetnyak's theorem holding on a large class of spaces. This class includes all Hadamard spaces, Banach spaces and complete Busemann spaces.
\begin{Thm}
\label{theorem1.1}
Let $X$ be a metric space admitting a contracting barycenter map. If $\eta:S^1 \rightarrow X$ is $L$-Lipschitz, then there exists an $L$-Lipschitz extension $f:H^2\rightarrow X$ of $\eta$ where $H^2$ is a metric hemisphere with boundary circle $S^1$.
\end{Thm}
If $X$ is a Hadamard space, then Theorem~\ref{theorem1.1} is a special case of a well known theorem of Urs Lang and Viktor Schroeder, \cite[Theorem A]{LS97}. Other Lipschitz extension theorems for target spaces of nonpositive curvature have been obtained in \cite{LPS00} and \cite{BS01}. Traditionally Lipschitz extensions have been studied in Banach space theory and Theorem~\ref{theorem1.1} is especially interesting in the setting that $X$ is a Banach space. There is a powerful method of proving Lipschitz extension results via barycentric constructions designed by James R. Lee and Assaf Naor in \cite{LN05}. Their method was developed further in \cite{Oht09} and \cite{AP19} and the refined variant will play a key role in the proof of Theorem~\ref{theorem1.1}.\par 
For a metric space $X$ denote \textit{Wasserstein $1$-space} over $X$ by $\mathcal{P}_1(X)$, see section~\ref{section2.1}. A map $b$ assigning to $\mu \in \mathcal{P}_1(X)$ a point $b(\mu)\in X$ is called \textit{barycenter map} if every Dirac measure $\delta_x$ one has  $b(\delta_x)=x$. The map $b:\mathcal{P}_1(X)\rightarrow X$ is called \textit{contracting} if it is $1$-Lipschitz with respect to \textit{Wasserstein $1$-distance} $d_W$.\par 
If $X$ is a Banach space one may define a contracting barycenter map simply via $b(\mu):=\int_X x \ \tr \mu(x)$ and if $X$ is a Hadamard space by minimizing the functional $q\mapsto \int_X d^2(p,q)\tr\mu(p)$. It turns out that contracting barycenter maps have a more geometric equivalent which are conical bicombings introduced by Dominic Descombes and Urs Lang in \cite{DL15}. A \textit{conical (geodesic) bicombing} $\sigma$ on $X$ is a map assigning to every tupel of points $(x,y)$ in $X$ a geodesic $\sigma_{x,y}$ such that for any pair of tupels $(x,y),(x',y')$ the distance function between $\sigma_{x,y}$ and $\sigma_{x',y'}$ satisfies a weak convexity condition, see section~\ref{section2.2}. A complete metric space admits a contracting barycenter map iff it admits a conical bicombing, see \cite[Theorem~3.4]{Bas18}. Conical bicombings are much easier to construct explicitly than contracting barycenter maps. Spaces admitting conical bicombings include all normed spaces, $\tn{CAT}(0)$ spaces, Busemann spaces, Wasserstein $1$-spaces and injective spaces in the sense of \cite{Lan13}.
\subsection{Applications}
An \textit{area functional} $\mathcal{A}$  is a functional assigning to each Lipschitz map $f:E\rightarrow X$ where $E \subseteq \R^2$ is a Borel set and $X$ is a metric space a number $\mathcal{A}(f)\in [0,\infty]$ such that certain natural axioms are fulfilled, see section \ref{section4.1}. Most intuitive example to have in mind is the \textit{Busemann area functional} $\mathcal{A}^b$ given by the parametrized $2$-dimensional Hausdorff measure of $\tn{im}(f)$.\par 
Fix an area functional $\mathcal{A}$ and a metric space $X$. Let $\eta:\K \rightarrow X$ be a Lipschitz curve. We call a Lipschitz map on the closed unit disc $f:D^2 \rightarrow X$ a \textit{filling} of $\eta$ if it restricts to $\eta$ on $\partial D^2=\K$. Define the \textit{filling area of $\eta$} with respect to $\mathcal{A}$, denoted $\tn{Fill}_\mathcal{A}(\eta)$, to be the infimum of $\mathcal{A}(f)$ where $f$ ranges over all fillings of $\eta$. We say that $X$ satisfies a \textit{$C$-quadratic isoperimetric inequality} with respect to $\mathcal{A}$ if for all $L\geq 0$ and all Lipschitz curves $\eta:\K \rightarrow X$ of length $L$ one has
 \begin{equation}
 \tn{Fill}_\mathcal{A}(\eta)\leq C\cdot L^2.
 \end{equation} Quadratic isoperimetric inequalities in this sense have been investigated by Alexander Lytchak, Stefan Wenger and Robert Young in~\cite{LWYar} and \cite{LW18}.\par
A consequence of Reshetnyak's majorization theorem is that Hadamard spaces satisfy a Euclidean aka $\frac{1}{4\pi}$-quadratic isoperimetric inequality with respect to every area functional. Similarly from  Theorem~\ref{theorem1.1} we derive the following.
\begin{Thm}
\label{theorem1.2}
Let $\mathcal{A}$ be an area functional and $X$ a metric space.
If $X$ admits a contracting barycenter map, then $X$ satisfies a $\frac{1}{2\pi}$-quadratic isoperimetric inequality with respect to $\mathcal{A}$.
\end{Thm}
In general the constant $\frac{1}{2\pi}$ in Theorem~\ref{theorem1.2} is optimal even in case $X$ is a Banach space and $\mathcal{A}=\mathcal{A}^b$ or $\mathcal{A}$ is the Holmes-Thompson area functional~$\mathcal{A}^{ht}$. This is due to a theorem of Sergei Ivanov which implies a lower bound on the filling area of isometrically embedded circles, see \cite[Theorem~2]{Iva11}.\par 
By local comparison, Reshetnyak's majorization theorem and Theorem~\ref{theorem1.2} also imply the following.
\begin{Cor}
\label{corollary1.3}
Let $M$ a smooth manifold, $\mathcal{A}$ an area functional and $\delta>0$.
\begin{enumerate}[label={\arabic*.}]
\item If $F:TM\rightarrow \R$ is a continuous Finsler structure on $M$, then $(M,d_F)$ satisfies a $\left(\frac{1}{2\pi}+\delta\right)$-local quadratic isoperimetric inequality with respect to~$\A$. 
\item If $g:TM\rightarrow \R$ is a continuous Riemannian structure on $M$, then $(M,d_g)$ satisfies a $\left(\frac{1}{4\pi}+\delta\right)$-local quadratic isoperimetric inequality with respect to~$\A$.
\end{enumerate}
\end{Cor}
In \cite{LW17a} Alexander Lytchak and Stefan Wenger solved the disc type Plateau problem in proper metric spaces $X$ for 'nice' area functionals $\mathcal{A}$. The arising solutions are Sobolev maps $u: \D \rightarrow X$ and will be called \textit{$\mathcal{A}$-minimal discs}. See section~\ref{section5} and the references therein for more precise definitions. Applying Theorem~\ref{theorem1.2} and Corollary~\ref{corollary1.3} to the regularity theorems of \cite{LW17a} and \cite{LW17b} one obtains results like the following.
 \begin{Thm}
 \label{theorem1.4}
Let $M$ be a smooth manifold and $F$ a continuous Finsler structure on $M$ such that $X:=(M,d_F)$ is complete. If $u$ is a bounded $\mathcal{A}^b$-minimal disc in $X$, then $u$ admits a representative that is locally $\alpha$-Hölder continuous in the interior of~$D^2$ for all $\alpha<\frac{\pi}{8}$.
\end{Thm}
Theorem~\ref{theorem1.4} and similar results will be discussed in section~\ref{section5}. Besides applying to a larger class of spaces Theorem~\ref{theorem1.4} improves the results~\cite[Theorem~1.2]{OvdM14} and \cite[Theorem~1.1]{PvdM17} of Patrick Overath, Sven Pistre and Heiko von der Mosel in the sense that it gives a large and universal Hölder constant. As in their work one may generalize Theorem~\ref{theorem1.4} to certain nonreversible Finsler structures. These are nonreversible Finsler structures satisfying property $(GA2)$ discussed in \cite{PvdM17}.
\subsection{The proof}
If $X$ admits a conical bicombing, there is an obvious cone construction of a filling of a given Lipschitz curve. This works to prove Theorem~\ref{theorem1.2} for the Holmes-Thompson area functional $\mathcal{A}^{ht}$ as I showed in my master thesis. However it fails for $\mathcal{A}^b$. The fillings we get for the proof of Theorem~\ref{theorem1.2} from Theorem~\ref{theorem1.1} are way more complicated and do not directly make use of the conical bicombing on $X$.\par
Lemma~$3.5$ in \cite{Oht09} reduces the proof of Theorem~\ref{theorem1.1} to the following very special case.
\begin{Pro}
\label{proposition1.5}
There exists an isometric embedding $\iota:H^2\rightarrow \mathcal{P}_1\left(S^1\right)$ such that~$\iota(P)=\delta_P$ for all $P\in S^1$.
\end{Pro}
The first step in the proof of Proposition~\ref{proposition1.5} is to discuss possible isometric extensions of $\delta:\K \rightarrow \mathcal{P}_1(\K)$ to one single point $P\in H^2\fgebackslash \K$. We obtain the following surprising lemma.
\begin{Lem}
\label{lemma1.6}
Let $P\in H^2\fgebackslash \K$. There is a natural bijection $\Phi_P$ between the set of $\pi$-periodic $\nu \in \mathcal{P}_1(\K)$ and the set of $\mu \in \mathcal{P}_1(\K)$ such that $d_W(\mu,\delta_Q)=\Ds(P,Q)$ for all $Q \in \Hs$. It is given by $\Phi_P(\nu):=h_P \cdot \mathcal{H}^1_\K+(1-k_P) \cdot \nu$ where $h_P:\K \rightarrow \R$ is a continuous density and $k_P \in [0,1)$ is a constant.
\end{Lem}
To prove Proposition~\ref{proposition1.5} we define $\iota:H^2\rightarrow W^1(\K)$ via $\iota(P):=\Phi_P(\tn{Uni}(\K))$ where $\tn{Uni}(\K)$ denotes the uniform distribution on $\K$. To prove that $\iota$ defines indeed an isometric embedding is a bit technical. A formula for Wasserstein-$1$-distance on $\K$ developed in \cite{CM95} reduces it to some analytic estimates of distances and angles on $\mathbb{S}^2$. The reason that the proof of Proposition~\ref{proposition1.5} gets a bit involved at this point is probably that the embedding $\iota$ is highly nonunique and the construction hence not very canonical.
\subsection{Outline of the paper}
In \ref{section2.1} we  fix notations for Wasserstein-$1$-space and contracting barycenter maps. Furthermore we discuss basic properties as well as the reduction of Theorem~\ref{theorem1.1} to Proposition~\ref{proposition1.5}. In \ref{section2.2} we discuss the connection of contracting barycenter maps and conical bicombing and give examples. Section~\ref{section3} is dedicated to the proof of Proposition~\ref{proposition1.5}. In \ref{section3.1} first we prove Lemma~\ref{lemma1.6}. Then in \ref{section3.2} and \ref{section3.3} we perform the more technical part of the proof of Proposition~\ref{proposition1.5}. The topic of section~\ref{section4} are quadratic isoperimetric constants. In \ref{section4.1} we discuss area functionals and Jacobians. In \ref{section4.2} we proof Theorem~\ref{theorem1.2} and discuss its optimality. In \ref{section4.3} we have a look at continuous Finsler structures and local quadratic isoperimetric inequalities. Finally in section~\ref{section5} we study $\mathcal{A}$-minimal discs and give the proof of Theorem~\ref{theorem1.4} as well as similar results.

\section{Wasserstein space, barycenter maps \& conical bicombings}
\subsection{Wasserstein $1$-space and Barycenter maps}
\label{section2.1}
Let $(X,d)$ be a metric space and $\mathcal{P}(X)$ the set of separably supported probability measures on the Borel $\sigma$-algebra of $X$. Note that if $X$ is complete, then $\mathcal{P}(X)$ is nothing but the set of Radon probability measures on $X$ by Ulam's theorem, see \cite[Theorem 7.1.4]{Dud89}. For $\mu,\nu\in \mathcal{P}(X)$ a  measure $K\in \mathcal{P}(X\times X)$ is called \textit{coupling} of $\mu$ and $\nu$ if $\pi_{1*}(K)=\mu$ and $\pi_{2*}(K)=\nu$. Here $\pi_i:X\times X \rightarrow X$ are the respective coordinate projections and $(-)_*$ indicates push forward of measures. Denote the set of couplings of $\mu$ and $\nu$ by $\Pi(\mu,\nu)$. The \textit{Wasserstein $1$-distance} $\W$ on $\mathcal{P}(X)$ is defined via
\[
\W(\mu,\nu):=\inf_{K\in \Pi(\mu,\nu)} \int_{X\times X} d(x,y) \ \tr K(x,y).
\]
Besides the fact that it might take infinite values $\W$ defines a metric on $\mathcal{P}(X)$, see \cite{Kel85}. If $x \in X$ and $\mu \in \mathcal{P}(X)$, then $\Pi(\mu,\delta_x)=\{\mu \otimes \delta_x\}$ and hence
\begin{equation}
\label{equation(1)}
\W(\delta_x,\mu)=\int_X d(x,y)\ \tr \mu(y).
\end{equation} 
By \eqref{equation(1)} the Dirac map $\delta:X \rightarrow (\mathcal{P}(X),d_W)$ defines an isometric embedding. Denote the subspace of measures at finite Wasserstein-$1$-distance from $\delta(X)$ by $\mathcal{P}_1(X)$ and call the arising metric space $(\mathcal{P}_1(X),\W)$ \textit{Wasserstein-$1$-space} over $X$.
It turns out that $\mathcal{P}_1(X)$ is complete iff $X$ is complete, see \cite[Theorem 6.18]{Vil09}, and that the construction is functorial in the following sense.
\begin{Lem}[\cite{Oht09}, Lemma $2.1$]
\label{lemma2.1}
Let $X,Y$ be metric spaces and $f:X\rightarrow Y$ be $L$-Lipschitz. Then $f_*:\mathcal{P}_1(X)\rightarrow \mathcal{P}_1(Y)$ is welldefined and $L$-Lipschitz.
\end{Lem}
If $b:\mathcal{P}_1(X)\rightarrow X$ is a $1$-Lipschitz retraction for $\delta:X \rightarrow \mathcal{P}_1(X)$, then we call $b$ a \tb{contracting barycenter map} on $X$.
\begin{Ex}
$1.$ Let $X$ be a Banach space. For $\mu\in \mathcal{P}_1(X)$  set $b(\mu):=\int_X x \ \tr \mu (x)$. This Bochner integral is welldefined and  defines a barycenter map. To check $b$ is contracting let $K\in \Pi(\mu,\nu)$. Then
\begin{align}
||b(\mu)-b(\nu)||&=\Big|\Big|\int_{X\times X} x \ \tr K(x,y)-\int_{X\times X} y \ \tr K(x,y)\Big|\Big|\\
&\leq \int_{X\times X} ||x-y||\ \tr K(x,y)\rightarrow \W(\mu,\nu).
\end{align}
One may show that this is the only barycenter map on $X$. If $C\subseteq X$ is convex, then restricting $b$ to $\mathcal{P}_1(C)$ gives a contracting barycenter map on $C$ in case either $C$ is closed or $X$ is finitedimensional. See \cite[Proposition~3.5]{Bas18} for more details.\par 
$2.$ Let $X$ be a Hadamard space. For $\mu \in \mathcal{P}_1(X)$ let $b(\mu)\in X$ be the unique point where $y \mapsto \int_{X} d^2(x,y)\ \tr\mu(x)$ attains its minimum. Then $b: \mathcal{P}_1(X)\rightarrow X$ defines a contracting barycenter map, see \cite{Stu03}.
\end{Ex}\par 
The following reduction is an incarnation of \cite[Lemma $3.5$]{Oht09}
\begin{proof}[Proposition~\ref{proposition1.5} $\Rightarrow$ Theorem~\ref{theorem1.1}]
Let $b:\mathcal{P}_1(X)\rightarrow X$ be a contracting barycenter map on $X$, $\iota:H^2\rightarrow \mathcal{P}_1(\K)$ an isometric embedding extending $\delta$ and $\eta:\K\rightarrow X$ an $L$-Lipschitz curve. Then $f:=b\circ \eta_* \circ \iota$ is $L$-Lipschitz by Lemma~\ref{lemma2.1}  and for $P\in \K$ \[f(P)=b(\eta_*(\iota(P)))=b(\eta_*(\delta_P))=b(\delta_{\eta(P)})=\eta(P).\]
\end{proof}
As a byproduct of this proof we get that for Banach spaces~$X$ the extension operator $E:\tn{Lip}(\K,X)\rightarrow \tn{Lip}(H^2,X)$ given by theorem~\ref{theorem1.1} is linear and functorial in $X$.
\subsection{Conical bicombings}
\label{section2.2}
Let $X$ be a geodesic metric space. A \textit{conical (geodesic) bicombing} on $X$ is a map $\sigma:X\times X \times[0,1]\rightarrow X$, such that for every $x,y\in X$ the map $\sigma_{x,y}:=\sigma(x,y,-)$ is a constant speed shortest path connecting $x$ and $y$ and for all $x,y,x',y'\in X$, $t \in [0,1]:$ 
\begin{equation} 
d\left(\sigma_{x,y}(t),\sigma_{x',y'}(t)\right) \leq (1-t) d(x,x')+t d(y,y').
\end{equation}
\begin{Ex}
\begin{enumerate}[label={\arabic*.}]
\item Let $X$ be a normed space. A conical bicombing $\sigma$ on $X$ is given by $\sigma(x,y,t):=(1-t)x+ty$. This conical bicombing on $X$ is unique, see \cite[Theorem 1]{GM81}. Restriction of $\sigma$ to $C \times C \times [0,1]$ gives a conical bicombing on any convex subset $C$ of $X$.
\item Let $X$ be a $\tn{CAT}(0)$ space in the sense of \cite{BBI01}. Then the unique geodesic bicombing $\sigma$ on $X$ is conical. More generally uniquely geodesic spaces admitting conical bicombings are called \textit{Busemann spaces} and have been studied long before the notion of conical bicombings had been invented, see \cite{Pap13} and the references therein.
\item Let $X$ be a metric space. A conical bicombing $\sigma$ on $\mathcal{P}_1(X)$ is given by $\sigma(\mu,\nu,t):=(1-t)\cdot\mu+t\cdot\nu$. This is actually a special case of $1.$ as $\mathcal{P}_1(X)$ may be considered a convex subset of the free Banach space~$\mathcal{F}(X)$, see \cite{AP19}.
\item Let $X$ be an injective metric space as studied by John R. Isbell in \cite{Isb64}. Then a conical bicombing $\sigma$ on $X$ may be defined applying the universal property of $X$. See \cite[Lemma 2.1]{DL15} and the following remark therein for the details.
\end{enumerate}
\end{Ex}
I was told about the following equivalence by Giuliano Basso.
\begin{Thm}
\label{theorem2.4}
Let $X$ be a metric space.
\begin{enumerate}[label={\arabic*.}]
\item \label{theorem2.4.1}
If $X$ admits a contracting barycenter map, then $X$ admits a conical bicombing.
\item \label{theorem2.4.2}
If $X$ is complete and admits a conical bicombing, then $X$ admits a contracting barycenter map.
\end{enumerate}
\end{Thm}
\begin{proof}
\tb{$1.:$} Let $b$ be a contracting barycenter map on $X$. Define the conical bicombing on $X$ by $\sigma(x,y,t):=b((1-t)\delta_x+t\delta_y)$.\par
\tb{$2.:$} By~\cite[Proposition 1.3]{BMar} there exists a reversible conical bicombing on $X$. So one may apply \cite[Theorem~3.4]{Bas18} to obtain the result.
\end{proof}
Note that in general neither conical bicombings nor barycenter maps need to be unique, see \cite[Example 3.4]{DL15}.
\section{Embedding $H^2$ into $\mathcal{P}_1(\K)$}
\label{section3}
\subsection{Extension to one single point.}
\label{section3.1}
In this paragraph we prove Lemma~\ref{lemma1.6}. First we investigate the analytic properties of the distance function $d_P:\K \rightarrow \R$, $Q\mapsto \Ds(P,Q)$ for a fixed $P \in H^2\fgebackslash \K$.
\begin{Lem}
\label{lemma3.1}
Let $P \in H^2\fgebackslash \K$ and $B \in \K$ such that $d_P(B)=\Ds(B, \K)$. Parametrize $\K$ by $(-\pi,\pi]$ such that $0$ corresponds to $B$. Then $d_P$ is smooth and for $t \in (-\pi,\pi]$
\begin{enumerate}[label={\arabic*.}]
\item $d_P'(t)\leq 0$ if $t\leq 0$ and $d_P'(t)\geq 0$ if $t\geq 0$.
\item \label{lemma3.1.2}
$d_P''(t)\geq 0$ if $|t|\leq \frac{\pi}{2}$ and $d_P''(t)\leq 0$ if $|t|\geq \frac{\pi}{2}$.
\item \label{lemma3.1.3}
$d_P'''(t)\geq 0$ if $t\leq 0$ and $d_P'''(t)\leq 0$ if $t\geq 0$.
\item \label{lemma3.1.4} 
$d_P'(t+\pi)=-d_P'(t)$  and $d_P''(t+\pi)=-d_P''(t)$.
\item \label{lemma3.1.5}
$d_P'(-t)=-d_P'(t)$ and $d_P''(-t)=d_P''(t)$.
\end{enumerate}
\end{Lem}
Here and in the following we consider $\K$ as boundary circle of the standard Riemannian hemisphere $H^2$. Apparently there is a natural $\R$ action on $\K$ by orientation preserving isometries that we denote by $+$. When we speak of parametrizations of $\K$, then we always mean orientation preserving unit speed parametrizations. Derivatives of functions defined on $\K$ such as $d_P$ are to be understood with respect to such parametrizations. When it does not lead to confusion we identify points of $\K$ and in the parametrizing interval of $\R$. 
\begin{proof}
Even though most parts of Lemma~\ref{lemma3.1} admit geometric arguments for sake of shortness we do the analytic computation. By the spherical cosine theorem $d_P(t)=\arccos(k_P \cos(t))$ where $k_P:=\cos(\Ds(P,B))$. So differentiating
\begin{align}
\label{equation(2)}
d_P'(t)&=\frac{k_P \sin(t)}{\sqrt{1-k^2_P\cos^2(t)}}\\
d_P''(t)&=\frac{(k_P-k^3_P)\cos(t)}{(1-k_P^2\cos^2(t))^{\frac{3}{2}}}\\
d_P'''(t)&=\frac{(k_P-k_P^3)(-1-2k_P^2\cos^2(t))}{(1-k_P^2\cos^2(t))^{\frac{5}{2}}}\sin (t).
\end{align}
As $k_P-k_P^3\geq 0$ and $-1-2k_P^2\cos^2(t)\leq 0$ this implies all the claims.
\end{proof}
We prove the following refined version of Lemma~\ref{lemma1.6}.
\begin{Lem}
\label{lemma3.2}
Let $P \in H^2 \fgebackslash \K$, $\mu \in \mathcal{P}_1\left(\K\right)$ and $T$ the antipodal map of $\K$. Then the following are equivalent:
\begin{enumerate}[label={\arabic*.}]
\item 
\label{lemma3.2.1}
For all $Q \in \K$ we have $\W(\mu, \delta_Q)=d_P(Q)$.
\item
\label{lemma3.2.2}
For all $Q \in \K$
\begin{equation}
\mu([Q,Q+\pi))=\Ha-\Ha d_P'(Q).
\end{equation}
\item 
\label{lemma3.2.3}
There exists $\nu \in \mathcal{P}_1(\K)$ such that $T_* \nu=\nu$ and \[
\mu=\Ha (d_P'')^+ \cdot \mathcal{H}^1_{\K}+(1-k_P)\cdot\nu.
\]
\end{enumerate}
\end{Lem}
Lemma~\ref{lemma1.6} follows from Lemma~\ref{lemma3.2} by setting $h_P:=\Ha (d_P'')^+$.
\begin{proof}
Let $B\in \K$ such that $\Ds(P,\K)=\Ds(P,B)$ and parametrize $\K$ such that $B$ corresponds to $0$. Then by Lemma~\ref{lemma3.1}.\ref{lemma3.1.2} \begin{equation}
\label{equation(3)}
(d_P'')^+=\mathbbm{1}_{(-\Pih,\Pih)}d_P''.
\end{equation} \par 
We take the left derivative of \eqref{equation(1)}, use the dominated convergence theorem and the fact that $\mu$ is a probability measure to get
\begin{align}
\frac{\partial_- \left(\W(\mu,\delta_Q)\right)}{\partial Q}&=\int_\K \frac{\partial_- \left(\Dk (Q,R)\right)}{\partial Q} \tr \mu(R)\\
&=\mu([Q-\pi,Q))-\mu([Q,Q+\pi))\\
\label{equation(4)}
&=1-2\mu([Q,Q+\pi)).
\end{align}
\ \par 
$\ref{lemma3.2.1}\Rightarrow \ref{lemma3.2.2}:$ By \eqref{equation(4)} $d_P'(Q)=1-2\mu([Q,Q+\pi))$ for all $Q \in \K$.
Solving this for $\mu([Q,Q+\pi))$ gives \ref{lemma3.2.2}.\par
$\ref{lemma3.2.2}\Rightarrow \ref{lemma3.2.3}$: Set $\overline{\nu}:=\mu-\Ha (d_P'')^+\mathcal{H}^1_\K$ which is a priori a finite signed measure on $\K$. First we show that $T_*\overline{\nu}=\overline{\nu}$. The spherical intervals of length $\leq \pi$ form a $\pi$-system generating the Borel $\sigma$-algebra of $\K$. So by Dynkins $\pi$-$\lambda$-theorem it suffices to check $T_*\overline{\nu}([t,s))=\overline{\nu}([t,s))$ for all intervals $[t,s)\subset \R$ of length $\leq \pi$. As $T_*\overline{\nu}([t,s))=\overline{\nu}([t+\pi,s+\pi))$ by symmetry we may assume $|t|\leq \Pih$. Then
\begin{align}
&T_*\overline{\nu}([t,s))-\overline{\nu}([t,s))=\overline{\nu}([s,s+\pi))-\overline{\nu}([t,t+\pi))\\
\label{equation(5)}
=&\underbrace{\mu([s,s+\pi))-\mu([t,t+\pi))}_{=:(X)}+\Ha \underbrace{\left(\int^{t+\pi}_t (d_P''(r))^+\tr r-\int^{s+\pi}_{s}(d_P''(r))^+ \tr r\right)}_{=:(Y)}.
\end{align}
We calculate $(X)$ and $(Y)$ separately.
\begin{equation}
\label{equation(6)}
(X)=\Ha-\Ha d_P'(s)-\left(\Ha-\Ha d_P'(t)\right)=\Ha d_P'(t)- \Ha d_P'(s)
\end{equation}
and 
\begin{equation}
\label{equation(7)}
(Y)=d_P'\left(\Pih\right)-d_P'(t)-\left(\begin{cases}
d_P'\left(\Pih\right)-d_P'(s) &,|s|\leq \Pih\\
d_P'(s+\pi)-d_P'\left(-\Pih\right) &, |s|>\Pih
\end{cases}\right)
=d_P'(s)-d_P'(t)
\end{equation}
where we used \eqref{equation(3)} and Lemma~\ref{lemma3.1}.\ref{lemma3.1.4}. Plugging \eqref{equation(6)} and \eqref{equation(7)} into \eqref{equation(5)} proves $T_*\overline{\nu}=\overline{\nu}$. So $\overline{\nu}$ is $\pi$-periodic and by \eqref{equation(3)} $\overline{\nu}_{|[\Pih,\frac{3\pi}{2})}=\mu_{|[\Pih,\frac{3\pi}{2})}\geq 0$. Hence $\overline{\nu}$ is a positive measure. Furthermore
\begin{equation}
\overline{\nu}(\K)=\mu(\K)-\Ha\int^{\frac{\pi}{2}}_{-\frac{\pi}{2}}d_P''(t) \ \tr t=1-d_P'\left(\Pih\right)=1-k_P
\end{equation}
where we used that \eqref{equation(2)} implies $k_P=d_P'\left(\Pih\right)$. Setting $\nu:=\frac{1}{1-k_P}\overline{\nu}\in \mathcal{P}_1(\K)$ completes this implication.\par
$\ref{lemma3.2.3}\Rightarrow \ref{lemma3.2.2}$: If $t \in \left(-\frac{\pi}{2},\frac{\pi}{2}\right]$ then
\begin{align}
\mu([t,t+\pi))&=(1-k_P)\nu([t,t+\pi))+\Ha \int^{\frac{\pi}{2}}_t d_P''(t) \ \tr t &&\\
&=\frac{1-k_P}{2}+\Ha d_P'\left(\frac{\pi}{2}\right)-\Ha d_P'\left(t\right) && \tn{as } T_*\nu=\nu \\
&=\Ha-\Ha d_P'(t) && \tn{as }k_P=d_P'\left(\Pih\right)
\end{align}
Going to complements and using Lemma~\ref{lemma3.1}.\ref{lemma3.1.4} provides the case $t \notin \left(-\frac{\pi}{2},\frac{\pi}{2}\right]$.\par 
$\ref{lemma3.2.2}+\ref{lemma3.2.3}\Rightarrow \ref{lemma3.2.1}$:
\eqref{equation(4)} and $\ref{lemma3.2.2}$ imply that
$\frac{\partial_-}{\partial Q}\W(\mu,\delta_Q)=d_P'(Q)$ for all $Q \in \K$. Furthermore the functions $d_P$ and $Q \mapsto \W(\mu,\delta_Q)$ are $1$-Lipschitz functions $\K \rightarrow \R$.
So by the fundamental theorem of calculus it suffices to check $\W(\mu,\delta_B)=d_P(B)$. By~\eqref{equation(1)}
\begin{equation}
\label{equation(8)}
\W(\mu,\delta_B)=\Ha\int^{\frac{\pi}{2}}_{-\frac{\pi}{2}}|t| d_P''(t) \tr t +(1-k_P)\int_\K d_\K(B,R) \ \tr \nu(R).
\end{equation}
We calculate the terms separately
\begin{align}
&\Ha\int^{\frac{\pi}{2}}_{-\frac{\pi}{2}}|t| d_P''(t) \tr t=\int^{\Pih}_0 td_P''(t)\tr t=[td_P'(t)]_{t=0}^{\Pih}-\int^{\Pih}_0 d_P'(t) \ \tr t\\
\label{equation(9)}
=&\Pih d_P'\left(\Pih\right)-d_P\left(\Pih\right)+d_P(0)=\Pih k_P-\Pih+d_P(B).
\end{align}
and
\begin{align}
&\int_\K d_\K(B,R)\tr \nu(R)=\int_\K d_\K(B,R)\tr T_*\nu(R)=\int_\K d_\K(B,T(R))\tr \nu(R)\\
=&\int_\K (\pi-d_\K(B,R))\tr \nu(R)=\pi-\int_\K d_\K(B,R)\tr \nu(R).
\end{align}
So $\int_\K d_\K(B,R)\tr \nu(R)=\Pih$. Plugging this and \eqref{equation(9)} into \eqref{equation(8)} gives the desired equality $\W(\mu,\delta_B)=\Ds(P,B)$.
\end{proof}
\subsection{Proof of Proposition~\ref{proposition1.5}}
\label{section3.2}
Let $X$ be a metric space. A variant of the Kantorovich-Rubinstein duality states that for $\mu,\nu \in \mathcal{P}_1(X)$ one has
\begin{equation}
\label{equation(10)}
\W(\mu,\nu)=\sup_{f} \left(\int_X f(x)\tr \mu(x)-\int_X f(x)\tr \nu(x)\right).
\end{equation}
where the supremum in \eqref{equation(10)} is taken over all $1$-Lipschitz functions $f:X \rightarrow \R$, see \cite[Theorem 2]{Kel85}.\par
For $X=\R$ the maximizers in~\eqref{equation(10)} are given in terms of the distribution functions $F_\mu$ and $F_\nu$ as precisely those $1$-Lipschitz $f:\R \rightarrow \R$ satisfying $f'=\tn{sgn}(F_\nu-F_\mu)$ almost everywhere on $\{F_\mu \neq F_\nu\}$ with respect to $\Leb^1$, see \cite[Corollary 2.2]{CM95}.\\
For $X=\K$ the issue is slightly more delicate. Call $C \in \K$ a \textit{balanced cut point} for~$(\mu,\nu)$ with Borel partition $\K= M \cupdot N$ if $\mathcal{H}_\K^1(M)=\mathcal{H}_\K^1(N)=\pi$ and 
\begin{align}
\mu([C,R))&\leq \nu([C,R)) &&;\forall R \in M\\
\mu([C,R))&\geq \nu([C,R)) &&;\forall R \in N.
\end{align}
\begin{Pro}[\cite{CM95}, Propositions~$3.2$ \& $3.6$]
\label{proposition3.3}
Let $\mu,\nu \in \mathcal{P}_1\left(\K\right)$. Then there exists a balanced cutpoint for $(\mu,\nu)$ and for every balanced cutpoint $C$ with Borel partition $\K=M\cupdot N$ one has
\[
\W(\mu, \nu)=\int_\K f(R) \ \tr (\mu-\nu)R
\]
where $f: \K \rightarrow \R$ is $1$-Lipschitz and such that
\[
f'=\mathbbm{1}_{M}-\mathbbm{1}_{N}
\]
holds $\mathcal{H}^1_\K$-almost everywhere.
\end{Pro}
Proposition~\ref{proposition3.3} prepares us to proof proposition~\ref{proposition1.5}.
\begin{proof}[Proof of Proposition~\ref{proposition1.5}]
Set $\iota: H^2\rightarrow P_1(\K)$ via $\iota(P):=\delta_P$ if $P \in \K$ and otherwise $\iota(P):=\mu_P$ where $\mu_P$ has density
\[
h_P:=\Ha (d_P'')^++\frac{1-k_P}{2\pi}
\] 
with respect to $\mathcal{H}^1_\K$. By Lemma~\ref{lemma1.6} the distances to the boundary points are preserved under $\iota$. So as $H^2$ is a length space it suffices to show that the restriction of $\iota $ to $H^2\fgebackslash  \K$ is distance preserving along geodesics.\par 
Let $l$ be an oriented minimizing geodesic segment in $H^2$ with starting point $C$ on $\K$, and let $\gamma\in [0,\pi]$ be the angle enclosed by $l$ and $\K$ in $C$. 
Let $A,A' \in l$ and without loss of generality $\Ds(A',C)\leq \Ds(A,C)$. Then by Lemma~\ref{lemma3.2}
\begin{equation}
\W(\mu_A,\mu_{A'})\geq \W(\mu_A,\delta_C)-\W(\mu_{A'},\delta_C)=\Ds(A,C)-\Ds(A',C)=\Ds(A,A').
\end{equation}
and hence $\iota$ is distance nondecreasing.\par 
The proof that $\mu$ is $1$-Lipschitz is more involved. Assume $\gamma\neq \Pih$. Parametrize $l$ by $a\in [0,\pi]$ as follows. Set $B_a:=C+a\in \K$ and let $A_a\in H^2$ be the intersection point of $l$ and the geodesic orthogonal to $\K$ in $B_a$. Setting $\mu_a:=\mu_{A_a}$ we have to show that for all $a_1,a_2\in(0,\pi)$.
\begin{equation}
\label{equation(11)}
\W(\mu_{a_1},\mu_{a_2})\leq \Ds(A_{a_1},A_{a_2}).
\end{equation}
\ \ \ Assume $\gamma <\Pih$ and $a_2,a_1 \leq \Pih$. Set $\mu_i:=\mu_{a_i}$, $A_i:=A_{a_i}$, $B_i:=B_{a_i}$, $d_i:=d_{A_i}$, $k_i:=k_{A_i}$ and let $h_i:=\Ha (d_i'')^++\frac{1-k_i}{2\pi}$ be the density of $\mu_i$ with respect to $\mathcal{H}^1_\K$ for $i=1,2$. Assume without loss of generality $a_2< a_1$ and hence $k_2 > k_1$. For $t \in \R$ let $C_t:=C+t\in \K$. If $C=C_0$ is a balanced cutpoint for $(\mu_1,\mu_2)$ with Borel partition $\K=[C,C_\pi)\cupdot [C_\pi,C_{2\pi})$, then by Proposition~\ref{proposition3.3} $d_C$ is a maximizer for \eqref{equation(10)} and hence by \eqref{equation(1)} and Lemma~\ref{lemma3.2}
\begin{align}
\W(\mu_1,\mu_2)&=\int_\K d_\K(Q,C)\tr(\mu_1-\mu_2)Q=\W(\mu_1,\delta_C)-\W(\mu_2,\delta_C)\\
&=\Ds(A_1,C)-\Ds(A_2,C)=\Ds(A_1,A_2).
\end{align}
So to get \eqref{equation(11)} in our special case $\gamma<\Pih$ and $a_1,a_2\leq \Pih$ it suffices to prove
 \begin{align}
\mu_{1}([C,C_t))&\leq \mu_{2}([C,C_t)) &&\forall t \in [0,\pi]\\
\label{equation(12)}
\mu_{1}([C,C_t))&\geq \mu_{2}([C,C_t)) &&\forall t \in [\pi,2\pi].
\end{align}
By Lemma~\ref{lemma3.2} and the first variation formula 
\begin{equation}
\label{equation(13)}
\mu_1([C,C_\pi))=\Ha-\Ha d_{1}'(C)=\Ha+\Ha \cos(\gamma)=\mu_2([C,C_\pi))
\end{equation}
and hence as $\mu_i$ are probability measures also $\mu_1([C_\pi,C_{2\pi}))=\mu_2([C_\pi,C_{2\pi}))$. Using this facts it follows that the following system of inequalities is equivalent to~\eqref{equation(12)}
\begin{align}
\label{equation(14)}
\mu_{1}([C,C_t))&\leq \mu_{2}([C,C_t)) &&\forall t \in \left[0,a_2+\Pih\right]\\
\label{equation(15)}
\mu_{1}([C_t,C_\pi)&\geq \mu_{2}([C_t,C_\pi)) &&\forall t\in \left[a_2+\Pih,\pi\right]\\
\label{equation(16)}
\mu_{1}([C_t,C))&\leq \mu_{2}([C_t,C)) &&\forall t \in \left[a_1-\Pih,0\right]\\
\label{equation(17)}
\mu_1([C_{-\pi},C_t))&\geq \mu_2([C_{-\pi},C_t)) && \forall t \in \left[-\pi,a_1-\Pih\right].
\end{align}
We prove \eqref{equation(14)}, \eqref{equation(15)}, \eqref{equation(16)} and \eqref{equation(17)} one by one.\par 
\eqref{equation(14)}: For all $0\leq s \leq t$ one has $h_i(C_s)=\Ha d_i''(C_s)+\frac{1-k_i}{2\pi}$. Hence by integration and the first variation formula
\begin{align}
\mu_i([C,C_t))&=\Ha d_i'(C_t)-\Ha d_{i}'(C)+t\frac{1-k_i}{2\pi}\\
&=-\Ha \cos(\gamma_{i,t})+\Ha \cos(\gamma)+t\frac{1-\cos(c_i)}{2\pi}.
\end{align}
where $c_i=c_{a_i}:=\Ds(A_i,B_i)$ and $\gamma_{i,t}=\gamma_{a_i,t}$ is the angle between the geodesic through $A_i$ and $\K$ in $C_t$. \begin{figure}[ht]
	\centering
		\includegraphics[scale=0.7]{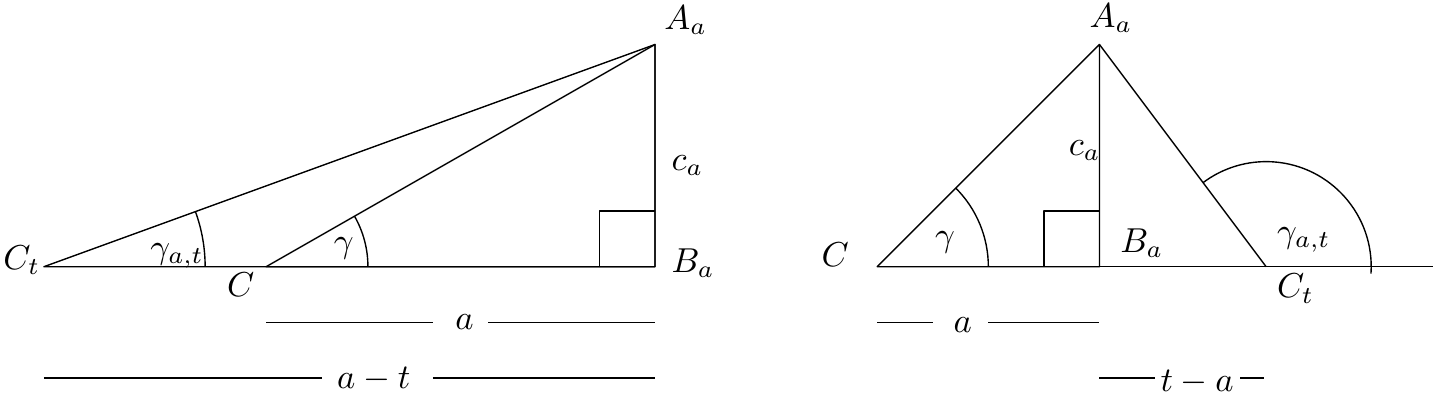}
	\caption[MapH]{Left: $-\Pih \leq t \leq 0$, Right: $a\leq t \leq \Pih$}
\end{figure}
\\
So
\begin{equation}
\mu_{2}([C,C_t))-\mu_{1}([C,C_t))=\Ha g(a_1,t)-\Ha g(a_2,t)
\end{equation}
where $g=g_\gamma:\left(0,\Pih\right]\times [-\pi,\pi]\rightarrow \R$ is defined by
\begin{equation}
g(a,t):=\cos(\gamma_{a,t})+\frac{t\cos (c_a)}{\pi}.
\end{equation}
So~\eqref{equation(14)} is implied by the following technical lemma whose proof we postpone to~\ref{section3.3}.
\begin{Lem}
\label{lemma3.4}
Let $\gamma<\Pih$ and $g=g_\gamma$ be defined as above. Then
\begin{enumerate}[label={\arabic*.}]
\item \label{lemma3.4.1} for $t \in [0,\pi]$ is $g(-,t)$ nondecreasing on $\left[\max\left\{0,t-\frac{\pi}{2}\right\},\frac{\pi}{2}\right]$.
\item \label{lemma3.4.2} for $t \in \left[-\frac{\pi}{2},0\right]$ is $g(-,t)$ nonincreasing on $\left[0,\frac{\pi}{2}\right]$.
\end{enumerate}
\end{Lem}
\eqref{equation(15)}: For $t \leq s \leq \pi$ one has
\begin{equation}
h_2(C_s)=\frac{1-k_2}{2\pi}\leq \frac{1-k_1}{2\pi}\leq h_1(C_s)
\end{equation}
which by integration immediately implies \eqref{equation(15)}.\par
\eqref{equation(16)}: By a calculation very analog to the proof of \eqref{equation(14)}
\begin{equation}
\mu_2([C_t,C))-\mu_1([C_t,C))=\Ha(g(a_2,t)-g(a_1,t))
\end{equation}
which is nonnegative by Lemma~\ref{lemma3.4}.\ref{lemma3.4.2}.
\par
\eqref{equation(17)}:
$h_1$ is constant on $I:=\left[C_{-\pi},C_{a_1-\Pih}\right]$ and by Lemma~\ref{lemma3.1}.\ref{lemma3.1.3} $h_2$ is nondecreasing on $I$. Furthermore $h_2(C_{-\pi})=\frac{1-k_2}{2\pi}<\frac{1-k_1}{2\pi}=h_1(C_{-\pi})$.\\
Assume $\mu_2([C_{-\pi},C_t))> \mu_1([C_{-\pi},C_t))$. Then $h_2(s)>h_1(s)$ for all $s \in \left[t,a_1-\Pih\right)$ which implies \[\mu_2\left(\left[C_{-\pi},C_{a_1-\Pih}\right)\right)>\mu_1\left(\left[C_{-\pi},C_{a_1-\Pih}\right)\right).\]
 So by \eqref{equation(13)} 
 \[\mu_2\left(\left[C_{a_1-\frac{\pi}{2}},C\right)\right)<\mu_1\left(\left[C_{a_1-\frac{\pi}{2}},C\right)\right)\]
 which contradicts \eqref{equation(16)}.
\par
So far we proved \eqref{equation(11)} for $\gamma < \Pih$ and $a_1,a_2\leq \Pih$. The other cases can be reduced to this one exploiting the fact that if $T$ is an isometry of $\Hs$ and $S$ its restriction to $\K$ then $S_*$ is an isometry of $\mathcal{P}_1(\K)$ and $S_*\mu_A=\mu_{T(A)}$.\par
For $\gamma>\Pih$ and $a_1,a_2\leq \Pih$ reflect in the geodesic passing through $C$ and orthogonal to $\K$ and we are in the previous case. For $\gamma\neq\Pih$ and $a_1,a_2\geq \Pih$ reflect in the geodesic passing through $C_{\Pih}$ and orthogonal to $\K$ and we are in a previous case. For $\gamma \neq \Pih$ and $a_2\leq \Pih\leq a_1$ by previous cases \eqref{equation(11)} applies for $a_1,\Pih$ and $\Pih,a_2$ respectively and hence by the triangle inequality also for $a_1,a_2$. Case $\gamma=\Pih$ maybe obtained by observing that it is nongeneric and applying the triangle inequality again.

\end{proof}
\subsection{A technical lemma}
\label{section3.3}
In this subsection we perform the proof of Lemma~\ref{lemma3.4}. We remind the reader of one of Napier's rules stating that for a nondegenerate spherical triangle $\triangle ABC$ with side length $a,b,c$
and angles $\alpha, \beta, \gamma$ satisfying $\beta=\Pih$ one has
\begin{equation}
\label{equation(18)}
\cot(\gamma)=\sin(a)\cot(c).
\end{equation}
\eqref{equation(18)} may be deduced directly from the spherical sine and cosine theorems.
\begin{proof}[Proof of Lemma~\ref{lemma3.4}]
We need to express $\cos(\gamma_{a,t})$ and $\cos(c_a)$ analytically in terms of $a,\gamma,t$. For $c_a$ by \eqref{equation(18)}
\begin{equation}
\label{equation(19)}
\cos(c_a)=\frac{\cot(c_a)}{\sqrt{1+\cot^2(c_a)}}=\frac{\frac{\cot(\gamma)}{\sin(a)}}{\sqrt{1+\frac{\cot^2(\gamma)}{\sin^2(a)}}}=\frac{\cot(\gamma)}{\underbrace{\sqrt{\sin^2(a)+\cot^2(\gamma)}}_{=:K(a)}}
\end{equation}
Applying \eqref{equation(18)} to $\bigtriangleup A_a B_a C$ and $\bigtriangleup A_a B_aC_t$
\begin{equation}
\cot^2(\gamma_{a,t})=\sin^2(a-t)\cot^2(c_a)=\frac{\sin^2(a-t)\cot^2(\gamma)}{\sin^2(a)}.
\end{equation}
Hence proceeding as in \eqref{equation(19)} and exploiting that $\tn{sgn}(\cos(\gamma_{a,t}))=\tn{sgn}(a-t)$
\begin{equation}
\label{equation(20)}
\cos(\gamma_{a,t})=\frac{\sin(a-t)\cot(\gamma)}{\underbrace{\sqrt{\sin^2(a)+\cot^2(\gamma)\sin^2(a-t)}}_{=:N(a,t)}}
\end{equation}
Careful differentiation of \eqref{equation(19)} and \eqref{equation(20)} gives
\begin{align}
\frac{\partial\cos(c_a)}{\partial a}&=\frac{-\cot(\gamma)\sin(a)\cos(a)}{K^3}\\
\frac{\partial \cos(\gamma_{a,t})}{\partial a}&=\frac{\cos(a-t)\cot(\gamma)\sin^2(a)-\sin(a)\cos(a)\sin(a-t)\cot(\gamma)}{N^3}\\
&=\frac{\cot(\gamma)\sin(a)\sin(t)}{N^3}.
\end{align}
And hence
\begin{equation}
\label{equation(21)}
\frac{\partial g}{\partial a}(a,t)=\underbrace{\cot(\gamma)\sin(a)}_{\geq 0}\left(\frac{\sin(t)}{N^3}-\frac{t\cos(a)}{\pi K^3}\right)
\end{equation}
\ \par
Case $t\in [0,\pi]$ \& $a \in \left[\max\left\{0,t-\Pih\right\},\Pih\right)$: As $N\leq K$  and $t \cos(a)\geq 0$ we estimate \eqref{equation(21)} to get
\begin{equation}
\label{equation(22)}
\frac{\partial g}{\partial a}(a,t)\geq \frac{\cot(\gamma)\sin(a)}{N^3}\left(\sin(t)-\frac{t\cos(a)}{\pi}\right).
\end{equation}
If $0\leq t \leq \frac{\pi}{2}$ then 
$\sin(t)-\frac{t\cos(a)}{\pi}\geq \sin(t)-\frac{t}{\pi}$ which is $\geq 0$ by concavity of the sine function on $\left[0,\Pih\right]$. And if $\Pih\leq t\leq \pi$ then as the cosine function is decreasing on~$\left[0,\Pih\right]$
\begin{equation}
\sin(t)-\frac{t\cos(a)}{\pi}\geq \sin(t)-\frac{t\cos\left(t-\Pih\right)}{\pi}=\sin(t)\left(1-\frac{t}{\pi}\right)\geq 0.
\end{equation}
Hence by \eqref{equation(22)} $\frac{\partial g}{\partial a}\geq 0$.\par
Case $t\in \left[-\Pih,0\right]$ \& $a \in \left[0,\Pih\right)$: Then as $N\leq K$ and $t\leq 0$
\begin{equation}
\frac{\partial g}{\partial a}(a,t)\leq \frac{\cot(\gamma)\sin(a)}{N^3}\left(\sin(t)-\frac{t\cos(a)}{\pi}\right)
\end{equation}
but $\sin(t)-\frac{t\cos(a)}{\pi}\leq\sin(t)-\frac{t}{\pi}\leq 0$.
\end{proof}
\section{Quadratic Isoperimetric Constants}
\label{section4}
\subsection{Volume functionals and Jacobians}
\label{section4.1}
In this paragraph we give a short introduction to the topic of volume functionals suited for our later purposes. The interested reader is referred for example to \cite{Iva09} and \cite{LW17b} for a more detailed exposition.\par
Let $n\in \N$. The $n$-dimensional \textit{Busemann volume functional} $\mathcal{V}^b$ assigns to a Lipschitz function $f: E \rightarrow X$ where $E \subseteq \R^n$ is a Borel set and $X$ a metric space the Hausdorff measure of the image counting multiplicities. Precisely
\begin{equation}
\mathcal{V}^b(f):=\int_X \tn{card}\left( f\inv (y) \right)\ \tr \mathcal{H}^n_X(y).
\end{equation}
If $X$ is a Riemannian manifold this seems quite adequate for measuring the $n$-dimensional volume of the image. However in more general situations the choice of Hausdorff measure on $X$ is less canonical and there exist other reasonable nonequivalent definitions.\par
Let $\mathcal{V}$ be a functional assigning to every Lipschitz map $f:E\rightarrow X$, where $E \subseteq \R^n$ is Borel and $X$ a metric space, a number $\mathcal{V}(f)\in [0,\infty]$. $\mathcal{V}$ is called~\textit{$n$-dimensional volume functional} if it satisfies the following properties for every such $f$:
\begin{enumerate}[label={\arabic*.}]
\item If $X$ is a Riemannian manifold, then $\mathcal{V}(f)=\mathcal{V}^b(f)$. \hfill (Normalization)
\item If $g:X\rightarrow Y$ is $L$-Lipschitz, then $\mathcal{V}(g\circ f)\leq L^n\cdot\mathcal{V}(f)$. \hfill (Monotonicity)
\item If $E'\subseteq \R^2$ is a Borel set and $\phi:E'\rightarrow E$ is bi-Lipschitz, then $\mathcal{V}(f\circ \phi)=\mathcal{V}(f)$.\par \hfill (Coordinate invariance)
\item If $M=\bigcupdot^\infty_{i=1}E_i$ where $E_i\subseteq \R^2$ are Borel sets, then
$\mathcal{V}(f)=\sum_{i=1}^\infty \mathcal{V}\left(f_{|E_i}\right)$. \par \hfill ($\sigma$-additivity)
\end{enumerate}
It is not hard to check that $\mathcal{V}^b$ is a volume functional in this sense. $2$-dimensional volume functionals will also be called \textit{area functionals} and will be denoted by $\mathcal{A}$ instead of $\mathcal{V}$.\par
Let $f:E \rightarrow X$ be a map where $E \subseteq \R^n$ is Borel and $X$ a metric space. $f$ is called \textit{metrically differentiable} at $p \in E$ if there exists a seminorm $s$ on $\R^n$ such that
\begin{equation}
\label{equation(23)}
\lim_{q \in E; q\rightarrow p} \frac{d(f(q),f(p))-s(q-p)}{|p-q|}=0.
\end{equation}
If this is the case such $s$ is called the \textit{metric differential} of $f$ at $p$ and is denoted by~$\tn{md}_pf$. If one replaces $\lim$ in \eqref{equation(23)} by the approximate limit $\tn{ap}\lim$ then one obtains the notions of \textit{approximate differentiability} and the \textit{approximate metric differential} $\textnormal{apmd}_p f$.\par
Bernd Kirchheim proved in \cite{Kir94} that if $f:E \rightarrow X$ is Lipschitz, then $f$ is metrically differentiable at almost every $p\in E$ and
\begin{equation}
\label{equation(24)}
\mathcal{V}^b(f)=\int_E \Ja^b(\tn{md}_p f)\tr\Leb^n(p).
\end{equation}
Here for a seminorm $s$ on $\R^n$ the $n$-dimensional Busemann Jacobian $J^b$ is defined via $\Ja^b(s):=\frac{\alpha_n}{\Leb^n(B_s)}$ where $B_s\subseteq \R^n$ is the unit ball of $s$ and $\alpha_n$ is the volume of the standard $n$-dimensional Euclidean unit ball. By \cite{Kar07} if $E=\bigcupdot^\infty_{i=0} E_i \cupdot S$ where $E_i$ are measurable such that the restrictions $f_{|E_i}$ are Lipschitz and $\Leb^2(S)=0$, then $f$ is approximately differentiable almost everywhere on $E$ and the analog of \eqref{equation(24)} holds with $\tn{md}_pf$ replaced by $\tn{apmd}_pf$, if furthermore $f$ satisfies Lusin's property~$(N)$.\par
The Busemann Jacobian may be generalized as follows.
Let $\Sigma^n$ be the set of seminorms on $\R^n$ and $\Sigma^n_0$ the set of norms on $\R^n$. An \textit{$n$-dimensional Jacobian} is a map $\Ja:\Sigma^n \rightarrow [0,\infty)$ fulfilling the following properties:
\begin{enumerate}[label={\arabic*.}]
\item
$\Ja(s_e)=\Ja^b(s_e)$ if $s_e$ is an Euclidean norm. \hfill (Normalization)
\item
$\Ja(s) \geq \Ja(s')$ whenever $s\geq s'$. \hfill (Monotonicity)
\item
$\Ja(s \circ T)=|\det{T}|\ \Ja(s)$ for $T \in M_n(\R)$. \hfill (Transformation law)
\end{enumerate}
Again one can check that $J^b$ defines a Jacobian in this sense.\par
If $\Ja$ is an $n$-dimensional Jacobian then the corresponding $n$-volume functional~$\mathcal{V}^{\Ja}$ is defined via
\begin{equation}
\label{equation(25)}
\mathcal{V}^{\Ja}(f):=\int_E \Ja(\tn{md}_p f)\ \tr \Leb^n(p).
\end{equation}
Vice versa if $\mathcal{V}$ is an $n$-dimensional volume functional then one may define a Jacobian $\Ja^\mathcal{V}$ by $\Ja^{\mathcal{V}}(s):=0$ if $s \in \Sigma^n \setminus \Sigma^n_0$ and 
\[\Ja^\mathcal{V}(s):=\mathcal{V}\left(\tn{id}:(0,1)^n\rightarrow \left((0,1)^n,s\right)\right)\]
otherwise. Applying \cite[Lemma 4]{Kir94} one can show that the operations ${\Ja}^\bullet$ and $\mathcal{V}^\bullet$ are mutually inverse. Hence $n$-Jacobians and $n$-volume functionals are incarnations of the same class of objects.\par 
It is easy to create other examples of Jacobians than $\Ja^b$ and hence other volume functionals than $\mathcal{V}^b$. Here are the most important ones.
\begin{Ex} 
Let $s \in \Sigma$. If $s\in \Sigma^n \setminus \Sigma^n_0$ set $\Ja^\bullet(s):=0$. Otherwise
\begin{enumerate}[label={\arabic*.}]
\item the Holmes-Thompson Jacobian $\Ja^{ht}$ is defined via $\Ja^{ht}(s):=\frac{\Leb^n(B^*_s)}{\alpha_n}$. Here $B^*_s:=\{v \in \R^n|\langle v,w\rangle \leq 1 ;\forall w \in B_s\}$ is the dual unit ball of $s$.
\item the inscribed Riemannian Jacobian $\Ja^{ir}$ is defined via $\Ja^{ir}(s):=\frac{\alpha_n}{\Leb^n(L_s)}$. Here $L_s$ is the John ellipsoid of $B_s$. That is the ellipsoid contained in $B_s$ of maximal $\mathcal{L}^n$ measure.
\end{enumerate}
\end{Ex}
One can obtain the following comparison results for the presented Jacobians. 
\begin{Thm}
\label{theorem4.2}
\begin{enumerate}[label={\arabic*.}]
\item $\Ja^{ht}\leq \Ja^{b}$ and for $n=2$ one has $\inf_{s\in \Sigma^2_0} \frac{\Ja^{ht}(s)}{\Ja^{b}(s)}=\frac{8}{\pi^2}$.
\item If $\Ja$ is an $n$-dimensional Jacobian, then $\Ja\leq \Ja^{ir}$ and \[q_n^{\Ja}:=\inf_{s \in \Sigma^n_0} \frac{\Ja(s)}{\Ja^{ir}(s)}\geq\left(\frac{1}{\sqrt{n}}\right)^n.\]
\item $q_2^b=\frac{\pi}{4}
$, $q_2^{ht}=\frac{2}{\pi}$ and $q_2^{ir}=1$.
\end{enumerate}
\end{Thm}
\begin{proof}
$\Ja^{ht}\leq \Ja^b$ is a formulation of the so called Blaschke-Santaló inequality, see \cite[$(10.28)$]{Sch13} and the references therein. The second part of $1.$ was proved in \cite{Mah39}. The optimal lower bound of $\frac{J^{ht}}{J^b}$ for 
general $n$ is subject to a famous conjecture of Kurt Mahler and remains open for $n \geq 4$. For $2.$ and $3.$ see \cite[Section $2.4$]{LW17b}.
\end{proof}
\subsection{Quadratic Isoperimetric Constants}
\label{section4.2}
Let $\mathcal{A}$ be an area functional and $X$ a metric space. If $\eta:\K \rightarrow X$ is a Lipschitz curve, then we call a Lipschitz map $f:\D\rightarrow X$ a \textit{filling} of $\eta$ if $f$ restricts to $\eta$ on $\K$. Here and in the following $D^2$ denotes the standard closed unit disc in $\R^2$. Define the \textit{filling area} of $\eta$ with respect to $\A$ by $\tn{Fill}_\mathcal{A}(\eta):=\inf_f \mathcal{A}(f)$ where $f$ ranges over all fillings of $\eta$. We say that $X$ satisfies a $C$-\tb{quadratic isoperimetric inequality} with respect to $\mathcal{A}$ if for all $\eta:\K \rightarrow X$ Lipschitz one has \[\tn{Fill}_\mathcal{A}(\eta)\leq C\cdot L_\eta^2,\] where $L_\eta$ denotes the length of $\eta$. 
Having set the definitions it is a straight forward consequence of Theorem~\ref{theorem1.1} that spaces admitting a contracting barycenter map satisfy a $\frac{1}{2\pi}$-quadratic isoperimetric inequality.
\begin{proof}[Proof of Theorem \ref{theorem1.2}]
Let $\A$ be an area functional, $\eta:\K \rightarrow X$ a Lipschitz curve of length $L_\eta=2\pi v$, $\overline{\eta}:\K\rightarrow X$ its constant speed parametrization and $\phi:\D\rightarrow H^2$ a diffeomorphism. Then $\overline{\eta}$ is of constant speed $v$ and hence $v$-Lipschitz. So by Theorem~\ref{theorem1.1} there exists an $v$-Lipschitz extension $f:H^2\rightarrow X$ of $\overline{\eta}$. Hence
\begin{equation}
\tn{Fill}_\mathcal{A}(\overline{\eta})\leq \mathcal{A}(f \circ \phi)\leq v^2 \mathcal{A}(\phi)=v^2 \mathcal{A}^b(\phi)=v^2 \mathcal{H}^2(H^2)=v^2 2\pi=\frac{1}{2\pi}L_\eta^2.
\end{equation}
But by \cite[Lemma~$3.5$]{LWYar} one has $\tn{Fill}_\A(\eta)=\tn{Fill}_\mathcal{A}(\overline{\eta})$.
\end{proof}
By a similar calculation invoking Reshetnyak's majorization theorem instead of Theorem~\ref{theorem1.1} the constant $\frac{1}{2\pi}$ in Theorem~\ref{theorem1.2} may be improved to $\frac{1}{4\pi}$ if $X$ is a Hadamard space. Conversely Alexander Lytchak and Stefan Wenger proved in \cite{LW18} that a proper metric space satisfying a $\frac{1}{4\pi}$-quadratic isoperimetric inequality with respect to $\mathcal{A}^b$ is a Hadamard space. This does not hold for $\mathcal{A}^{ht}$ as by \cite[Theorem 4.4.2]{Tho96} all two dimensional normed spaces satisfy a $\frac{1}{4\pi}$-quadratic inequality with respect to $\mathcal{A}^{ht}$. The following theorem of Sergei Ivanov implies that the constant $\frac{1}{2\pi}$ in Theorem~\ref{theorem1.2} can in general not be improved even for Banach spaces and $\mathcal{A}=\mathcal{A}^{ht}$.
\begin{Thm}[\cite{Iva09}, Theorem~$5.2$ \& \cite{Iva11}, Theorem~$2$]
\label{theorem4.3}
Let $X$ be a metric space and $\iota:\K\rightarrow X$ an isometric embedding. Then $\tn{Fill}_{\mathcal{A}^{ht}}(\iota)\geq 2\pi$
\end{Thm}
By Theorem~\ref{theorem4.3} and Theorem~\ref{theorem4.2} if a space $X$ admits an isometric embedding of $\K$ then Theorem~\ref{theorem1.2} cannot be improved for $X$ and $\mathcal{A}^{ht}/\mathcal{A}^{b}/\mathcal{A}^{ir}$. Examples of such embeddings are Kuratowski embedding $\K \rightarrow L^\infty(\K)$ and $\delta: \K \rightarrow \mathcal{P}_1(\K)$.\par 
\subsection{Local Quadratic Isoperimetric Inequality Constants}
\label{section4.3}
Our setting is the following. Let $M$ be a smooth manifold. A \textit{continuous Finsler/Riemannian structure} on $M$ is a continuous function $F:TM\rightarrow \R$ such that $F$ is continuous and for every $p\in M$ the restriction $F_p:=F_{|T_pM}:T_pM \rightarrow \R$ is a norm/an Euclidean norm. We call the tuple $(M,F)$ a \textit{continuous Finsler/Riemannian} manifold. For a piecewise differentiable curve $\eta:[a,b]\rightarrow M$ set
\begin{equation}
L^F_\eta:=\int^b_a F_{\eta(t)}\left(\eta'(t)\right)\tr t.
\end{equation}
$L^F$ induces a metric on $M$ which we denote by $d_F$. For $x,y\in M$ it is given by
\begin{equation}
d_F(x,y):=\inf \left\{L^F_\eta| \eta \tn{ piecewise smooth, }\eta(a)=x,\eta(b)=y\right\}.
\end{equation}
Locally $(M,d_F)$ is close to a normed space in the following sense.
\begin{Lem}
\label{lemma4.4}
Let $(M,F)$ be a continuous Finsler manifold, $p \in M$ and $\epsilon >0$. Then there exists a neighborhood $U\subseteq M$ of $p$, a convex set $V \subseteq \R^n$, a norm $s\in \Sigma_0^n$ on $\R^n$ and a diffeomorphism $\phi:V \rightarrow U$ such that $\phi:(V,s)\rightarrow (U,d_F)$ is $(1+\epsilon)$-bi-Lipschitz. If $(M,F)$ is even a continuous Riemannian manifold, then $s$ is a Euclidean norm.
\end{Lem}
The proof of Lemma~\ref{lemma4.4} is a bit tedious but essentially straight forward. Compare for example \cite[Section 3.2]{LY06}.\par
Let $\mathcal{A}$ be an area functional and $X$ a metric space. We say that $X$ satisfies a $C$-\tb{local quadratic isoperimetric inequality} with respect to $\mathcal{A}$ if for every $p \in X$ there exists a neighborhood $U$ of $p$ that satisfies a $C$-quadratic isoperimetric inequality with respect to $\A$. So by Theorem~\ref{theorem1.1} and Lemma~\ref{lemma4.4} we get Corollary~\ref{corollary1.3}.
\begin{proof}[Proof of Corollary~\ref{corollary1.3}]
We prove $1.$ The proof of $2.$ is similar. Let $p\in M$, $\epsilon >0$ arbitrary and choose $U,C,s,\phi$ as in Lemma~\ref{lemma4.4}. Let $\eta: \K \rightarrow U$ be a Lipschitz curve. Then by Theorem~\ref{theorem1.2}
\begin{equation}
\tn{Fill}_\A(\eta)\leq (1+\epsilon)^2\tn{Fill}_\A(\phi\inv\circ \eta)\leq \frac{(1+\epsilon)^2}{2\pi} L_{\phi\inv\circ\eta}^2\leq  \frac{(1+\epsilon)^4}{2\pi}L_\eta^2.
\end{equation}
As $\epsilon>0$ was arbitrary this completes the proof.
\end{proof} 
By generalizing all the invoked definitions and theorems in an obvious way one may see that also an adequate variant of Corollary~\ref{corollary1.3}.$1$ holds for nonreversible Finsler structures.
\section{Regularity of area minimizers}
\label{section5}
In this section we discuss how our results can be applied to the work of Alexander Lytchak and Stefan Wenger in \cite{LW17a} and \cite{LW17b} to obtain regularity results for minimal discs.\par 
Let $X$ be a complete metric space $X$ and $p>1$. Essentially following \cite{Res97} the Sobolev space $ \mathcal{W}^{1,p}(\D,X)$ may be defined as the set of maps $f:\D\rightarrow X$ satisfying: \begin{enumerate}[label={\arabic*.}]
\item There exists $w \in L^p(\D)$ such that for every $g:X\rightarrow \R$ $1$-Lipschitz the composition $g\circ f$ belongs to the classical Sobolev space $\mathcal{W}^{1,p}(\D,\R)$ and $|\nabla(g \circ f)|\leq w$ almost everywhere.
\item $f$ is essentially separably valued. That is there exists $N\subset \D$, $\Leb^2(N)=0$ such that $f(D^2\fgebackslash N)$ is separable.
\end{enumerate}
Let $u \in \mathcal{W}^{1,p}(\D,X)$. Then by \cite[2.7.]{Kar07} one may decompose $\D=\bigcupdot_{i=1}^\infty E_i\cupdot S$ such that $u$ is Lipschitz restricted to $E_i$ and $\Leb^2(S)=0$. So  $u$ is approximately metrically differentiable almost everywhere and if $\Ja$ is a Jacobian one may define $\A^\Ja(u)$ as in \eqref{equation(25)} with $\tn{md}_p u$ replaced by $\tn{apmd}_p u$.
Furthermore as in the situation of the classical Sobolev spaces one may naturally define a trace map $\tn{tr}(u):\K\rightarrow X$ representing the boundary of $u$, see \cite[Section 1.12]{KS93}.\par
So there are all the ingredients to define filling areas and quadratic isoperimetric inequality in a Sobolev sense. For $\eta:\K \rightarrow X$  Lipschitz we call $u \in W^{1,2}(\D,X)$ a \textit{Sobolev filling} of $\eta$ if $\tn{tr}(u)$ parametrizes $\eta$. Furthermore if $\mathcal{A}$ is an area functional we define the \textit{Sobolev filling area} of $\eta$ by $\tn{Fill}^S_\mathcal{A}(\eta):=\inf \A(u)$ where $u$ ranges over all Sobolev fillings of $\eta$. $X$ is said to satisfy a \textit{$C$-Sobolev quadratic isoperimetric inequality} w.r.t. $\mathcal{A}$ if for all $\eta:\K \rightarrow X$ Lipschitz $\tn{Fill}^S_\A(\eta)\leq C\cdot L_\eta^2$.\par
Lipschitz maps $\D\rightarrow X$ are especially contained in $\mathcal{W}^{1,2}(\D,X)$ and for such the trace is nothing but restriction to the boundary. Hence $\tn{Fill}^S_\mathcal{A}\leq \tn{Fill}_\A$. By \cite[Proposition 3.1]{LWYar} equality holds if $X$ is Lipschitz $1$-connected and hence by Theorem~\ref{theorem1.1} especially for spaces admitting a contracting barycenter map. So for such spaces the notions of $C$-quadratic isoperimetric inequality and $C$-Sobolev quadratic isoperimetric inequality are equivalent.\par
In \cite{LW17a} Lytchak and Wenger proved that if $X$ is proper, $\A$ is \textit{quasiconvex} and $\eta$ is a Jordan curve such that $\tn{Fill}_\A^S(\eta)<\infty$, then there exists a Sobolev filling $u$ of $\eta$ such that $\A(u)=\tn{Fill}^S_\A(\eta)$. Furthermore $u$ may be chosen minimizing the Reshetnjak energy among all such minimal $\A$-fillings. We call such $u$ minimizing area and energy \tb{$\A$-minimal disc} following the terminology of \cite{LW17b}.
Quasiconvexity of $\mathcal{A}$ means that flat $2$-dimensional discs in finite dimensional normed spaces are minimal fillings of their boundary curves with respect to $\mathcal{A}$. This holds for all discussed examples of area functionals, see \cite[Theorem 3]{BI02}, \cite[Theorem 1]{BI12}, \cite[Theorem 6.2]{Iva09} respectively. In \cite{GWar} the existence of $\A$-minimal discs for quasiconvex area functionals and prescribed Jordan curve was generalized to a larger class of metric spaces including apart proper spaces also Hadamard spaces, dual Banach spaces and injective spaces by Chang-Yu Guo and Stefan Wenger.\par
Besides other regularity properties of $\A$-minimal discs Lytchak and Wenger proved a variant of the following.
\begin{Thm}
\label{theorem5.1}
Let $X$ be a complete metric space satisfying a $C$-Sobolev local quadratic isoperimetric inequality, $\A$ an area functional and $u$ an $\A$-minimal disc. \\ If $u(\D)$ is precompact in $X$, then $u$ admits a representative $\overline{u}$ that is locally $\alpha$-Hölder continuous on the interior of $\D$ where
$\alpha=\frac{q_2^\mathcal{A}}{4\pi C}$. If $X$ satisfies property $(ET)$ this improves to $\alpha=\frac{1}{4\pi C}$.
\end{Thm}
Following \cite[Section 11]{LW17a} a space is said to satisfy property $(ET)$ if for all $u \in \mathcal{W}^{1,2}(\D,X)$ almost all approximate metric differentials $\tn{apmd}_p u$ are possibly degenerate Euclidean seminorms. Examples of spaces satisfying property $(ET)$ include continuous Riemannian manifolds and spaces admitting a lower or upper curvature bound in the sense of Alexandrov. The precompactness assumption in Theorem~\ref{theorem5.1} e.g. applies automatically in case $X$ is compact, $u$ is continuous or $X$ is proper and $u$ bounded.
\begin{proof}[Proof of Theorem~\ref{theorem5.1}]
This is a variant of \cite[Theorem 4.5]{LW17b}. However unfortunately their formulation does not completely fit to our situation. Instead of a local quadratic isoperimetric inequality and precompact image, they demand a quadratic isoperimetric inequality at uniform small scales. However the quadratic isoperimetric inequality property of $X$ only comes into play when proving \cite[Lemma 8.6]{LW17a}. By Lebesgue covering theorem curves of uniform small scale in $u(\D)\subseteq X$ satisfy a $C$-quadratic isoperimetric inequality. This suffices to perform the vary same proof of \cite[Lemma 8.6]{LW17a}.\par 
The constant $q_2^{\A}$ only comes into play when estimating $\A^{ir}(u_{|\Omega})$ by $\frac{1}{q_2^\A}\A(u_{|\Omega})$ where $\Omega \subseteq \D$. So due to the normalization property $q_2^\A$ may be replaced by $1$ if $X$ satisfies property $(ET)$.
\end{proof}
We may apply Theorem~\ref{theorem1.1} and Corollary~\ref{corollary1.3} to calculate concretely Hölder regularity constants of $\mathcal{A}$-minimal discs via Theorem~\ref{theorem5.1}.\par 
Let $X$ be a complete metric space, $\mathcal{A}$ an area functional and $u \in \mathcal{W}^{1,2}(\D,X)$ $\mathcal{A}$-minimal with precompact image.  Then $u$ admits a representative that is locally $\alpha$-Hölder continuous in the interior of $\D$...
\begin{enumerate}[label={\arabic*.}]
\item ... for $\alpha=\frac{q_2^{\mathcal{A}}}{2}$ if $X$ admits a conical geodesic bicombing. Especially for $\mathcal{A}=\mathcal{A}^{ht}/\mathcal{A}^{b}/\mathcal{A}^{ir}$ one may take $\alpha=\frac{1}{\pi}\big/\frac{\pi}{8}\big/ \frac{1}{2}$ respectively. If $X$ furthermore satisfies $(ET)$, then $\alpha$ improves to $\Ha$ for every $\mathcal{A}$.
\item ... for $\alpha=1$ if $X$ is a Hadamard space. So in this situation $u$ is locally Lipschitz in the interior. For the Reshetnyak energy replaced by the Korevaar-Schoen-energy this follows also from \cite{KS93}.
\item ... for all $\alpha<\frac{q_2^{A}}{2}$ if $X$ is a continuous Finsler manifold. Especially for $\mathcal{A}=\mathcal{A}^{ht}/\mathcal{A}^{b}/\mathcal{A}^{ir}$ for all $\alpha<\frac{1}{\pi}\big/\frac{\pi}{8}\big/ \frac{1}{2}$ respectively. As a special case this proves Theorem~\ref{theorem1.4}.
\item ... for all $\alpha<1$ if $X$ is a continuous Riemannian manifold or a space of nonpositive curvature in the sense of Alexandrov.
\end{enumerate}
For $m\in \N$ a not necessarily symmetric norm $s$ is said to satisfy $(GAm)$ if its $m$-harmonic symmetrization \[
s_{(m)}(v):=\left(\frac{2}{s(v)^{-m}+s(-v)^{-m}}\right)^{1/m}
\]
defines a norm. By \cite[Corollary 3.4]{PvdM17} one may see that also $\mathcal{A}^b$-minimal discs in nonreversible Finsler manifolds such that the norms on all tangent spaces satisfy $(GA2)$ are locally $\alpha$-Hölder continuous for all $\alpha <\frac{\pi}{8}$. 
\section{Acknowledgements}
This article extends my master thesis "Quadratic Isoperimetric Inequality Constants" handed in at the Universität zu Köln in January 2018. I want to thank my supervisor Alexander Lytchak for stating the topic and really great support during the time of writing the thesis and this article. Furthermore I would like to thank Christian Lange for checking the main proofs calculation and Giuliano Basso for a useful discussion.\par 
Also I would like to thank the anonymous referee for pointing out several typos and giving minor corrections.
\bibliographystyle{alpha}
\bibliography{ms}
\end{document}